\documentclass[11pt]{amsart}

\usepackage[colorlinks=true,citecolor=black!60!green,linkcolor=black!60!red,filecolor=black!60!cyan,urlcolor=black!60!magenta]{hyperref}
\usepackage[text={6.1in,8.5in},centering]{geometry}
\usepackage{amssymb,amsmath,amsthm,mathtools,extpfeil}
\usepackage[all,cmtip]{xy}
\usepackage{tikz}
\usepackage{caption,subcaption}
\usepackage[shortlabels]{enumitem}
\usepackage{hyperref}
\usepackage{comment}

\newtheorem{innercustomthm}{Theorem}
\newenvironment{refthm}[1]
  {\renewcommand\theinnercustomthm{#1}\innercustomthm}
  {\endinnercustomthm}

\newtheorem{thm}{Theorem}
\newtheorem*{thm*}{Theorem}
\newtheorem*{lem*}{Lemma}
\newtheorem{lem}[thm]{Lemma}

\newtheorem*{prop*}{Proposition}

\newtheorem{conj}[thm]{Conjecture}
\theoremstyle{definition}

\newtheorem{rmk}[thm]{Remark}

\numberwithin{thm}{section}

\DeclareMathOperator{\rank}{rank}
\DeclareMathOperator{\Area}{area}
\DeclareMathOperator{\vol}{vol}
\DeclareMathOperator{\id}{id}
\newcommand{\ph}{\varphi}
\newcommand{\epsi}{\varepsilon}

\newcommand{\ZZ}{\mathbb{Z}}

\begin{document}
\title{On Freedman's link packings}
\author{Fedor Manin}
\address[F.~Manin]{Department of Mathematics, UCSB, Santa Barbara, CA, United States}
\email{manin@math.ucsb.edu}
\author{Elia Portnoy}
\address[E.~Portnoy]{Department of Mathematics, MIT, Cambridge, MA, United States}
\email{eportnoy@mit.edu}
\begin{abstract}
    Recently, Freedman introduced the idea of packing a maximal number of links into a bounded region subject to geometric constraints, and produced upper bounds on the packing number in some cases, while commenting that these bounds seemed far too large.  We show that the smallest of these ``extravagantly large'' bounds is in fact sharp by constructing, for any link, a packing of exponentially many copies as a function of the available volume.  We also produce improved and generalized upper bounds.
\end{abstract}
\maketitle

\section{Introduction}

\subsection{Main results}
In \cite{Free}, Freedman proposed a new type of quantitative problem in topology.  Given a link $L \subset [0,1]^3$, define the \emph{$\epsi$-diagonal packing number} $n_L(\epsi)$ to be the number of copies of $L$ that can be simultaneously embedded in $[0,1]^3$ subject to the following constraints:
\begin{enumerate}
    \item \textbf{Topological constraint:} Each copy is contained in a disjoint ball.
    \item \textbf{Geometric constraint:} Within a copy, the components are separated by a distance of at least $\epsi$.
\end{enumerate}
\noindent We refer to an embedding satisfying these constraints as an \emph{$\epsi$-diagonal embedding}.

This can be thought of, as Freedman suggests, as a topological analogue of the famous sphere-packing problem; it can also be thought of as a generalization of the study of ``geometric'' or ``thick'' knots and links, in which one may for example ask for the length of rope of a specified thickness required to tie a particular knot or link, or how this length relates to topological knot or link invariants.

Freedman showed that $n_L(\epsi)$ is $\exp(O(\epsi^{-3}))$ for links with nonzero linking number, and more generally $\exp(\operatorname{poly}(\epsi^{-1}))$ for homotopically nontrivial links in the sense of Milnor, but suggested that the true answer for all links may be $O(\epsi^{-3})$.

In this paper we disprove this suggestion, showing instead that Freedman's upper bound for the Hopf link is sharp:
\begin{thm} \label{lower-bound}
    For every tame link $L$, $n_L(\epsi) \geq \exp(C(L)\epsi^{-3})$, where $C(L)>0$ is a constant.
\end{thm}
\noindent We also improve and generalize Freedman's upper bounds using his methods and some additional ideas:
\begin{thm} \label{sharp-upper-bound}
    For every homotopically nontrivial link $L$, $n_L(\epsi)=\exp(O(\epsi^{-3}))$.
\end{thm}
\begin{thm} \label{weak-upper-bound}
    For every link with a nontrivial Milnor invariant of order $q$,
    \[n_L(\epsi)=\exp(O(\epsi^{-6(q-1)})).\]
\end{thm}
\noindent Here, as usual, the notation $O(f(x))$ refers to an arbitrary function which is bounded by $Cf(x)$ for some constant $C$; in each case, the constant depends on the link type.

For example, the Hopf link, the Borromean rings, and other Brunnian links are homotopically nontrivial, and therefore our results give exact asymptotics for them (up to a constant in the exponent).  The Whitehead link is an example of a homotopically trivial link with a nontrivial Milnor invariant of order $4$, so if $L$ is the Whitehead link we know that $n_L(\epsi)=\exp(O(\epsi^{-18}))$.  The methods of Theorem \ref{weak-upper-bound} can potentially be extended to some further links (for example, those which have a finite cover with a nontrivial Milnor invariant, as studied in \cite{COrr}) but this leaves many links for which we have no upper bound, for example any boundary link.

We do not know any lower bounds on the $\epsi$-diagonal packing number stronger than Theorem \ref{lower-bound}.  In fact, a slightly stronger hypothesis makes this bound sharp for any nontrivial link.  We say an embedding of a link is \emph{$\epsi$-thick} if the exponential map on the $\epsi/2$-neighborhood of the zero section of the normal bundle is still an embedding.  (Of course, such an embedding only exists for tame links.)  Define $n'_L(\epsi)$ to be the number of copies of $L$ that can be embedded in disjoint balls if every copy is required to be $\epsi$-thick.  The only difference from the original condition is that each component must be ``far away from itself''.  Then:
\begin{thm} \label{thick}
    For every tame link $L$ which is not the unlink, there are constants $C>c>0$ such that
    \[\exp(c\epsi^{-3}) \leq n'_L(\epsi) \leq \exp(C\epsi^{-3}).\]
\end{thm}
This suggests:
\begin{conj}
    For every link $L$ with nontrivial linking (that is, $L$ does not consist of a set of knots contained in disjoint balls), $n_L(\epsi)=\exp(O(\epsi^{-3}))$.
\end{conj}
If $L$ is a disjoint union of knots, then one can $\epsi$-pack an arbitrary number of copies of $L$ by making the components arbitrarily tiny and placing them in disjoint neighborhoods; this accounts for the difference between the conjecture and Theorem \ref{thick}.  If $L$ has nontrivial linking, then there is no such simple solution.  However, we have no further evidence that the conjecture holds for boundary links or for that matter the Whitehead link.  Thus the possibility remains that $n_L(\epsi)$ is larger or even infinite for some nontrivial links.

% We use the following notation for asymptotic behavior of functions:
% \begin{itemize}
% \item $f(x)=O(g(x))$ if there are $a$ and $b$ such that for large enough $x$, $f(x) \leq ag(bx)$.
% \item $f(x)=\Omega(g(x))$ if $g(x)=O(f(x))$.
% \item $f(x)=\Theta(g(x))$ if $f(x)=O(g(x))$ and $g(x)=O(f(x))$.
% \end{itemize}
% In particular, exponentials with different bases are equivalent under this relation.

\subsection{Relation to ropelength}
The \emph{ropelength} of a knot or link is, informally, the number of centimeters of $1$cm-thick rope required to tie it.  Formally, say that an embedding of a link is \emph{$\epsi$-thick} if the exponential map on the $\epsi/2$-neighborhood of the zero section of the normal bundle is still an embedding.  Then the ropelength of a link is the minimal total length of a $1$-thick embedding.  Ropelength and its relationship to invariants of knots and links has been studied by a number of researchers, for example in \cite{CKS, DEPZ, KM}.

In a similar vein, one can study the largest $\epsi$ such that a link has an $\epsi$-thick embedding into a unit cube.  By Weyl's tube formula, this invariant is bounded above by $(\text{ropelength})^{-1/3}$; conversely, it is clearly at least $(\text{ropelength})^{-1}$, though it is plausible that a better lower bound could be obtained.

The geometric conditions for diagonal $\epsi$-embedding can be thought of as a generalization of this condition in which only some components of the link are required to be separated.  In this context, the potential difference between diagonal $\epsi$-embedding and the variation in which each component must be $\epsi$-thick is particularly interesting.

\subsection{Proof ideas}
All of our upper bounds on $n_L(\epsi)$ and $n'_L(\epsi)$ use a trick also used by Freedman in \cite{Free}, as follows.  Fix an $\epsi$-diagonal embedding of $nL$, the union of $n$ copies of the $m$-component link $L$ contained in disjoint balls.  We tile the cubes by tiny cubes at scale $\sim \epsi$; then for each copy of $L$, the components are contained in disjoint sets of tiles.  We can think of this as giving a coloring of the tiling by $m+1$ colors, depending on which component, if any, intersects a given tile.  The number of possible colorings is $\exp(O(\epsi^{-3}))$.

In each case, for the proofs of Theorems \ref{sharp-upper-bound}, \ref{weak-upper-bound}, and the upper bound of Theorem \ref{thick}, the next step is to bound the number of copies of $L$ that can share the same coloring.  This determines the final bound.

The naturalness of this strategy is underscored by the construction of $\epsi$-diagonal embeddings that proves our lower bounds.  In this construction, we associate each copy of $L$ to a binary word of length $\sim \epsi^{-3}$, where each letter determines, for certain pairs of small cubes, which of the two has nontrivial intersection with a certain component of $L$.  In other words, in this construction, the map from copies of $L$ to colorings is injective.

\subsection{Acknowledgements}
The authors would like to thank Michael Freedman for encouraging them to write up these results, an anonymous referee for suggesting a simplification of the proof of Theorem \ref{sharp-upper-bound} and a number of other helpful comments, and Ben Chung, Merrick Hua, and Genevieve Yao for noticing several errors in previous drafts.  F.~M.\ would also like to thank Larry Guth and MIT for hosting him for a visit during which these many of these ideas were conceived.  Additionally, F.~M.\ was partially supported by NSF individual grant DMS-2204001 and a Sloan Fellowship.

\section{Exponential lower bound}

In this section, we construct diagonal $\epsi$-embeddings of exponentially many copies of a link, proving Theorem \ref{lower-bound} as well as the lower bound of Theorem \ref{thick}. Before describing the general construction for Theorem \ref{lower-bound}, we sketch a construction for the following weaker statement: For any number $N>0$ we can embed $2^{N^3}$ copies of the Hopf link $L$ into $100I^3 = [0,100]^3$ so that any two loops in the same copy of $L$ are at least distance $N^{-1}$ apart, and any two loops in different copies of $L$ are unlinked.

Consider a solid torus $U$, embedded in $100I^3$, whose cross sections have radius $3N^{-1}$ and whose longitude has length $N^2$. Divide the torus into $N^3$ disjoint boxes, each of which looks roughly like a cube of side length $N^{-1}$. Then divide each box into three sections, labeled $0, 1$ and $2$ so that sections $0$ and $2$ are distance roughly $N^{-1}$ apart in each box. Let $U'$ be a solid torus in $I^3$, disjoint from $U$, so that the longitudes of $U$ and $U'$ form a Hopf link. 

Begin by embedding $2^{N^3}$ copies of $L$ into $100I^3$ as follows. Label the copies of $L$ with $N^3$-letter binary words in the alphabet $\{0,2\}$. For a word $w$, let $w_j$ stand for its $j^{th}$ letter, and let $L^i_w$ be the $i^{th}$ strand in the copy of $L$ with label $w$. For each word $w$, let $L^1_w$ travel along the longitude of $U$ so that in the $j^{th}$ box of $U$, it passes through the section labeled $w_j$ and does not intersect the other sections. For each word $w$, let $L^2_w$ travel along the longitude of $U'$. All the copies of $L$ are now mutually linked, and we modify this embedding by adding tentacles to the curves $L^2_w$ to unlink certain pair of strands.

For each pair of distinct words $w, w'$ do the following modification. Pick an index $j$ such that $w_j \neq w'_j$. Band sum $L^2_w$ with a tentacle that only intersects $U$ in the $j^{th}$ box in section $w'_j$. In this section, the tentacle winds around $L^1_{w'}$ so as to unlink $L^1_{w'}$ from $L^2_w$. By making this tentacle thin enough we can ensure that it does not wind around any other strand in any other copy of $L$. The tentacle is distance at least $N^{-1}$ from $L^1_w$ because of the way it intersects $U$. Similarly, band sum $L^2_{w'}$ with a tentacle that only intersects $U$ in the $j^{th}$ box in section $w_j$ and winds around $L^1_w$ there. After this modification, $L^1_w$ and $L^2_w$ are still distance $N^{-1}$ apart, but now the $w$ and $w'$ copy of $L$ are unlinked. After completing all such modifications we get an $\epsi$-diagonal embedding of $2^{N^3}$ copies of $L$. This completes the sketch.

To go from pairwise unlinking the strands to placing all the links in disjoint balls, we would have to route the tentacles much more carefully, since most possible routings lead to complicated linking between groups of several strands. In the proof below, we give a method of routing strands which avoids this issue.

\begin{refthm}{\ref{lower-bound}}
  Given a tame link $L$, let $nL$ be the link consisting of $n$ copies of $L$ contained in disjoint balls.  Then there is a $C(L)>0$ such that for every $\epsi>0$, there is a diagonal $\epsi$-embedding of $nL$ with $n(\epsi)=\exp(C(L)\epsi^{-3})$.  Moreover, the embedding can be chosen so that each strand is $\epsi$-thick.
\end{refthm}

\begin{proof}
  The overall idea is as follows.  We start with a crossing diagram for $L$, and replace each strand with $n$ parallel strands.  This gives a link which satisfies the geometric condition for a diagonal $\epsi$-embedding, but is not isotopic to $nL$.  We then modify the link at each crossing in the diagram so that it is isotopic to $nL$ while maintaining the $\epsi$-separation between corresponding strands.  This can be done relatively easily because the crossings that need to be changed are between strands which are not required to be $\epsi$-separated.

  Now we give the details.  Start by choosing a diagram for $L$ and a specific realization of $L$ whose projection to the $xy$-plane is this diagram.  Let $L=\gamma^1 \cup \cdots \cup \gamma^s$.  Let $n=2^{(c/\epsi)^3}$, with $c=c(L)>0$ depending on other constants $c_j$ introduced throughout the proof, which will also implicitly depend on $L$.  We replace each $\gamma^i$ with copies $\gamma^i_w$, labeled by binary words of length $(c/\epsi)^3$, which are stacked vertically (that is, perpendicular to the plane of the diagram) within a $c_1$-neighborhood of $\gamma^i$.  The $\gamma^i_w$ are arranged in lexicographic order from top to bottom.

  Now we modify these curves within a $c_2$-neighborhood of each crossing between $\gamma^i$ and $\gamma^j$ (with $i$ and $j$ not necessarily distinct).  Our goal is to make the resulting link isotopic to a vertical stack of copies of $L$; we therefore need to flip every crossing in which $\gamma^i_w$ and $\gamma^j_{w'}$ are in the wrong order.

  Let $p$ be the point on $\gamma^i$ at which it crosses under $\gamma^j$, and consider a parallelepiped $P$ around $p$ of width and length $c_3$ and height $1$, with $c_3$ chosen so that $P$ is separated from all other crossings and strands of the link.  We divide $P$ into $r=(c/\epsi)^3$ sub-parallelepipeds $K_1,\ldots,K_r$ so that each $K_\ell$ is of dimensions at least $9\epsi \times 3\epsi \times 7\epsi$.  These are numbered boustrophedonically, so that $K_\ell$ and $K_{\ell+1}$ either share a face parallel to the $yz$-plane or are adjacent and on the exterior of the parallelepiped.

  Now fix a ribbon $R$ (i.e.~an embedding of a long thin rectangle) of width $5\epsi$ which traverses the $K_1,\ldots,K_r$ in order, such that $R \cap K_\ell$ lies in the central $xz$-plane of $K_\ell$, and such that its cross-section is also vertical everywhere.  Note that $R$ must exit $P$ to move between cubes with different $y$ or $z$ coordinates.  If we have chosen a reasonable ordering of the $K_\ell$ and a reasonable course for $R$ without gratuitous knots, then rerouting all the $\gamma^i_w$ and $\gamma^j_w$ in a neighborhood of the crossing so as to travel along $R$, while maintaining their vertical ordering, does not change the link type.

  In our final construction, the $\gamma^i_w$ and $\gamma^j_w$ will travel along $R$ or very close to it while maintaining the following properties:
  \begin{enumerate}
  \item \textbf{Geometric constraint:} Each curve will be $\epsi$-thick, and $\gamma^i_w$ will be $\epsi$-separated from $\gamma^j_w$.
  \item \textbf{Topological constraint:} The intersection of the link with a neighborhood of $P$ will be isotopic as a tangle (that is, relative to the boundary) to a vertical stack in which each $\gamma^j_w$ crosses over the corresponding $\gamma^i_w$, and all strands are arranged in lexicographic order from top to bottom.  Thus the $\gamma^j$ and $\gamma^i$ strands alternate vertically.
  \end{enumerate}
  We do this in two steps.  We will start by building a set of curves lying in $R$ and satisfying the geometric constraint.  In this set of curves, some curves will actually intersect each other.  We will then show that we can resolve these intersections by jostling the curves in such a way that the topological constraint is also satisfied.  This movement can be arbitrarily $C^1$-tiny and therefore the geometric constraint will still be satisfied.

  We start by fixing the geometry.  We assume that $R$ is sufficiently nice geometrically that the geometric constraint can easily be satisfied outside the $K_\ell$, where we only mandate that the curves are arranged vertically with the $\gamma^j_w$ in lexicographic order above and the $\gamma^i_w$ in lexicographic order below.  Inside the $K_\ell$, the behavior will depend on the $\ell$th letter of $w$, $w_\ell$.  Identify $K_\ell \cap R$ with $[0,9\epsi] \times [0,5\epsi]$.  If $w_\ell=0$, then $\gamma^j_w$ travels through $[0,9\epsi] \times [4\epsi,5\epsi]$ and $\gamma^i_w$ travels through $[0,9\epsi] \times [2\epsi,3\epsi]$.  On the other hand, if $w_\ell=1$, then $\gamma^j_w$ starts with its $z$-coordinate in $[4\epsi,5\epsi]$, travels in that band until $x \in [3\epsi,4\epsi]$, dips down until $z \in [2\epsi,3\epsi]$, travels right until $x \in [5\epsi,6\epsi]$, and then travels back up until $z \in [4\epsi,5\epsi]$ and exits with the rest of the $\gamma^j_w$.  Similarly, $\gamma^i_w$ lies in
  \[([0,2\epsi] \times [2\epsi,3\epsi]) \cup ([\epsi,2\epsi] \times [0,3\epsi]) \cup ([\epsi,8\epsi] \times [0,\epsi]) \cup ([7\epsi,8\epsi] \times [0,3\epsi]) \cup ([7\epsi,9\epsi] \times [2\epsi,3\epsi]).\]
  The $\gamma^j_w$ and $\gamma^i_w$ maintain their relative vertical orderings with respect to others with the same $\ell$th letter.  Moreover, while $x \in [4,5]$, the $\gamma^j_w$ with $w_\ell=1$ and $\gamma^i_w$ with $w_\ell=0$ are arranged in lexicographic order with respect to each other.  See Figure \ref{fig:crossingsLeft} for an illustration.

  If we haven't done anything silly, we now have curves satisfying the geometric constraint, but with nontrivial intersections: whenever $w' \neq w$, $\gamma^i_{w'}$ and $\gamma^j_w$ intersect each other twice for every $\ell$ such that $w(\ell)=1$ and $w'(\ell)=0$.  The $\gamma^j_w$ often intersect each other as well, as do the $\gamma^i_w$.  To finish the construction, we resolve all these intersections into crossings so as to satisfy the topological constraint.
  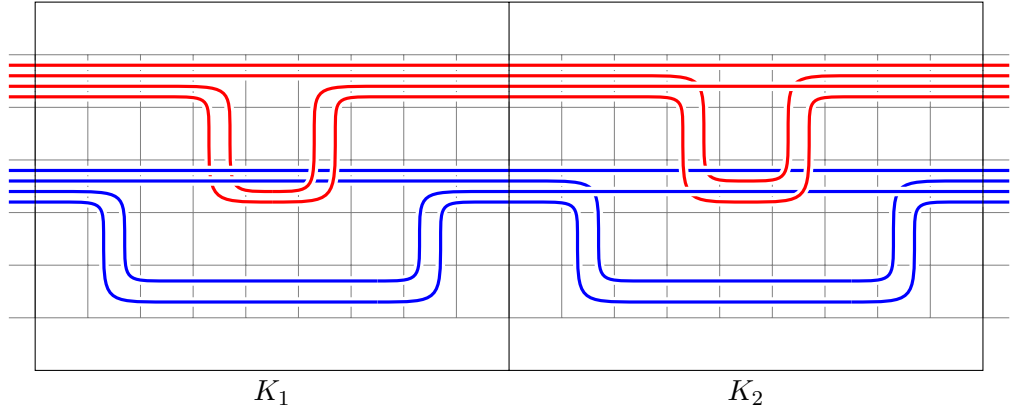
\begin{figure}
      \centering
      \begin{tikzpicture}[scale=0.7]
        \draw[help lines] (-0.5,1) grid (18.5,6);
        
        \tikzstyle{blueline}=[draw=white, double=blue, very thick, double distance=1.2pt];
        \tikzstyle{redline}=[draw=white, double=red, very thick, double distance=1.2pt];
        % \node[anchor=east] at (-0.5,7.5) {$\gamma^j_{00}$};
        % \node[anchor=east] at (-0.5,6.5) {$\gamma^j_{01}$};
        % \node[anchor=east] at (-0.5,5.5) {$\gamma^j_{10}$};
        % \node[anchor=east] at (-0.5,4.5) {$\gamma^j_{11}$};
        % \node[anchor=east] at (-0.5,3.5) {$\gamma^i_{00}$};
        % \node[anchor=east] at (-0.5,2.5) {$\gamma^i_{01}$};
        % \node[anchor=east] at (-0.5,1.5) {$\gamma^i_{10}$};
        % \node[anchor=east] at (-0.5,0.5) {$\gamma^i_{11}$};
        \node[anchor=north] at (4.5,0) {$K_1$};
        \node[anchor=north] at (13.5,0) {$K_2$};
        \draw [redline] (-0.5,5.2) -- (2.5,5.2) .. controls (3.3,5.2) .. (3.3,4.5) .. controls (3.3,3.2) .. (4.5,3.2);
        \draw [redline] (-0.5,5.4) -- (2.5,5.4) .. controls (3.7,5.4) .. (3.7,4.5) .. controls (3.7,3.4) .. (4.5,3.4);
        \draw [redline] (-0.5,5.6) -- (11.5,5.6) .. controls (12.7,5.6) .. (12.7,4.5) .. controls (12.7,3.6) .. (13.5,3.6);

        \draw [blueline] (-0.5,3.8) -- (18.5,3.8);
        \draw [blueline] (-0.5,3.6) -- (9.5,3.6) .. controls (10.7,3.6) .. (10.7,2.5) .. controls (10.7,1.7) .. (11.5,1.7) -- (15.5,1.7);
        \draw [blueline] (18.5,3.6) -- (17.5,3.6) .. controls (16.3,3.6) .. (16.3,2.5) .. controls (16.3,1.7) .. (15.5,1.7);
        \draw [redline] (9,5.2) -- (11.5,5.2) .. controls (12.3,5.2) .. (12.3,4.5) .. controls (12.3,3.2) .. (13.5,3.2);
        \draw [blueline] (-0.5,3.4) -- (0.5,3.4) .. controls (1.7,3.4) .. (1.7,2.5) .. controls (1.7,1.7) .. (2.5,1.7) -- (6.5,1.7);
        \draw [blueline] (18.5,3.4) -- (8.5,3.4) .. controls (7.3,3.4) .. (7.3,2.5) .. controls (7.3,1.7) .. (6.5,1.7);
        \draw [blueline] (-0.5,3.2) -- (0.5,3.2) .. controls (1.3,3.2) .. (1.3,2.5) .. controls (1.3,1.3) .. (2.5,1.3) -- (6.5,1.3);
        \draw [blueline] (9,3.2) -- (9.5,3.2) .. controls (10.3,3.2) .. (10.3,2.5) .. controls (10.3,1.3) .. (11.5,1.3) -- (15.5,1.3);
        \draw [blueline] (9,3.2) -- (8.5,3.2) .. controls (7.7,3.2) .. (7.7,2.5) .. controls (7.7,1.3) .. (6.5,1.3);
        \draw [blueline] (18.5,3.2) -- (17.5,3.2) .. controls (16.7,3.2) .. (16.7,2.5) .. controls (16.7,1.3) .. (15.5,1.3);
        
        \draw [redline] (-0.5,5.8) -- (18.5,5.8);
        \draw [redline] (18.5,5.6) -- (15.5,5.6) .. controls (14.3,5.6) .. (14.3,4.5) .. controls (14.3,3.6) .. (13.5,3.6);
        \draw [redline] (18.5,5.4) -- (6.5,5.4) .. controls (5.3,5.4) .. (5.3,4.5) .. controls (5.3,3.4) .. (4.5,3.4);
        \draw [redline] (18.5,5.2) -- (15.5,5.2) .. controls (14.7,5.2) .. (14.7,4.5) .. controls (14.7,3.2) .. (13.5,3.2);
        \draw [redline] (9,5.2) -- (6.5,5.2) .. controls (5.7,5.2) .. (5.7,4.5) .. controls (5.7,3.2) .. (4.5,3.2);
        \draw (0,0) -- (0,7) -- (18,7) -- (18,0) -- cycle (9,0) -- (9,7);
      \end{tikzpicture}
      \captionsetup{singlelinecheck=off}
      \caption[.]{The pattern of curves and crossings on the ribbon $R$ for words of length $2$.  From the top going down at either endpoint, we have:
      \[\gamma^j_{00}, \gamma^j_{01}, \gamma^j_{10}, \gamma^j_{11}, \gamma^i_{00}, \gamma^i_{01}, \gamma^i_{10}, \gamma^i_{11}.\]
      } \label{fig:crossingsLeft}
  \end{figure}
  \begin{figure}
      \centering
      \begin{tikzpicture}[scale=0.3]
        \tikzstyle{blueline}=[draw=white, double=blue, very thick, double distance=1.2pt];
        \tikzstyle{redline}=[draw=white, double=red, very thick, double distance=1.2pt];
        \draw [redline] (-0.5,7.5) -- (10.5,7.5);
        \draw [redline] (-0.5,4.5) .. controls (1,4.5) .. (1,2) arc(180:270:1) -- (8,1);
        \draw [redline] (-0.5,5.5) -- (0,5.5) .. controls (1.5,5.5) .. (1.5,3) arc(180:270:1) -- (7.5,2);
        \draw [redline] (-0.5,6.5) -- (0.5,6.5) .. controls (2,6.5) .. (2,4) arc(180:270:1) -- (7,3);
        \draw [blueline] (-0.5,3.5) -- (10.5,3.5);
        \draw [blueline] (-0.5,2.5) -- (10.5,2.5);
        \draw [blueline] (-0.5,1.5) -- (10.5,1.5);
        \draw [blueline] (-0.5,0.5) -- (10.5,0.5);
        \draw [redline] (10.5,4.5) .. controls (9,4.5) .. (9,2) arc(0:-90:1);
        \draw [redline] (10.5,5.5) -- (10,5.5) .. controls (8.5,5.5) .. (8.5,3) arc(0:-90:1);
        \draw [redline] (10.5,6.5) -- (9.5,6.5) .. controls (8,6.5) .. (8,4) arc(0:-90:1);
        %\node[white,anchor=north] at (2,0) {$K_1(3\epsi)$};
      \end{tikzpicture}
      \caption{By inspection, the tangle in Figure \ref{fig:crossingsLeft} is isotopic to this one.  Now think of the red strands as moving from back to front as they go left to right.}
      \label{fig:crossings}
  \end{figure}
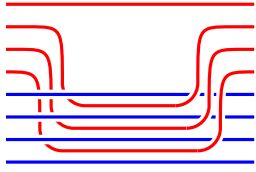
  \begin{lem}
      The intersections can be resolved into over- and undercrossings so that the resulting tangle inside the cube is equivalent via an isotopy fixing the boundary to the tangle $T_0(|w|)$ in which segments of $\gamma^j_w$ crossing back to front alternate vertically with segments of $\gamma^i_w$ crossing left to right, with the $w$ in lexicographic order from top to bottom.
  \end{lem}
  \begin{proof}
      We prove this by induction on the length of $w$.  The case $|w|=2$ is illustrated in Figure \ref{fig:crossings}.  Now suppose we have the lemma for $|w|=L$.  The tangle for words of length $L+1$ can be produced as follows.  Take two copies of the tangle for words of length $L$; call them $T$, identified with words of the form $0w$, and $T'$, identified with words of the form $1w$.  Place $T'$ in front of $T$, and connect this on the left (in $K_1$) with another tangle in which all strands $\gamma^j_w$ for $w$ starting with $1$ are routed under all strands $\gamma^i_{w'}$ for $w'$ starting with $0$.  The inductive step going from $|w|=1$ to $|w|=2$ can again be observed in Figure \ref{fig:crossingsLeft}.

      By induction, the right side is isotopic via an isotopy fixing the boundary to two copies of $T_0(L)$, one placed in front of the other.  After performing this isotopy, we can further isotope the full tangle to $T_0(L+1)$ using the following steps:
      \begin{enumerate}
          \item Move the copy of $T_0(L)$ corresponding to words starting with $1$ underneath the one corresponding to words starting with $0$, except for the rightmost part which must stay fixed.
          \item At the portion of the tangle around the boundary between $K_1$ and $K_2$, push the strands $\gamma^j_{1w}$ below all strands $\gamma^i_{0w'}$, so that they do not rise back up after dipping to their lowest point inside $K_1$.
      \end{enumerate}
      After these two steps, the tangle diagram is isomorphic to that of $T_0(L+1)$.
  \end{proof}
  Resolving the intersections can be done with arbitrarily $C^1$-small perturbations of the curves, and these preserve the geometric constraint.  Assuming we have given the curves $\gamma^i$ and $\gamma^j$ local orientations so that $\gamma^j$ crosses over $\gamma^i$ in a positive crossing, the lemma gives our tangle the desired topology.  Therefore this completes the construction.
\end{proof}

\section{Packing bounds for thick embeddings}

In this section we prove the upper bound of Theorem \ref{thick}.  That is, we show that the exponential bound for $n_L(\epsi)$ from the previous section is tight for any non-trivial link when we assume that every component has an $\epsi$-thick embedding.

\begin{lem} \label{knot_isotopy} Suppose $K_1$ and $K_2$ are two smoothly embedded knots which have $\epsi$-thick embeddings into $[0,1]^3 = I^3$. If $K_1$ lies in the $\epsi/4$-neighborhood of $K_2$, then there is an embedded annulus in $I^3$ with boundary $K_1 \cup K_2$ which is also contained in the $\epsi/4$-neighborhood of $K_2$.
\end{lem}

\begin{proof}
From the definition of a $\epsi$-thick embedding, we see that every point in an $\epsi/2$ neighborhood of $K_2$ has a unique nearest neighbor in $K_2$. So in this neighborhood there is a well-defined projection along the normal bundle of $K_2$ to the zero section of the normal bundle. Let $\pi: K_1 \to K_2$ be this projection on $K_1$. To find our annulus it suffices to show that $\pi$ is injective: then the linear homotopy from $K_1$ to $\pi(K_1)$ traces out an embedded annulus.

We will repeatedly use the following auxiliary lemma:
\begin{lem} \label{lem:connected}
  The intersection of $K_i$ with any open $\delta$-ball is connected, for $\delta \leq \epsi/2$.  Moreover, if a $\delta$-ball $B$ is tangent to $K_i$ at a point $q$, then $B \cap K_i$ is empty.
\end{lem}
\begin{proof}
    For the first statement, suppose not, and let $x$ be the center of the ball.  Then there are two points $p_1,p_2 \in K_i$ such that $d(p_i,x)<\delta$, and $f(p)=d(p,x)$ as a function $K_i \to [0,\infty)$ attains a local minimum at $p_i$.  Therefore $x$ is in the radius $\delta$ normal disk to $K_i$ around both $p_1$ and $p_2$, contradicting the thickness of $K_i$.

    For the second statement, we can apply the same argument in a ball shifted very slightly in the direction of $q$.  In this ball, $q$ is a local minimum of the distance to the center; if $B \cap K_i$ is nonempty, then we can find a second local minimum.
\end{proof}

\begin{figure}
  \centering
  \begin{tikzpicture}[scale=3]
    \clip (-1.5,-1.3) rectangle (1.5,0.3);
    \fill[color=gray!30!white] (0,-1) circle(0.5);
    \filldraw[fill=gray!30!white] (-1,-0.15) circle(1);
    \filldraw[fill=gray!30!white] (1,-0.15) circle(1);
    \draw (0,-1) circle(0.5);
    \draw (0,-1) circle(1);
    \draw (0,0) -- (0,-1);
    \draw (-1,-0.15) -- (1,-0.15);
    \draw[red, very thick] (0,0) arc(90:140:1.1) arc(-40:-80:1.5) (0,0) arc(-90:-70:2) node[anchor=north west] {$K_2$} arc(-70:-60:4);
    \draw[blue, very thick] (0,-0.15) arc(180:150:1.2) (0,-0.15) arc(180:190:1.9) node[anchor=south west] {$K_1$};
    \filldraw[black] (0,0) circle (0.5pt) node[anchor=south east]{$p$};
    \filldraw[black] (0,-0.15) circle (0.5pt) node[anchor=north east]{$q$};
    \filldraw[black] (0,-1) circle (0.5pt) node[anchor=east]{$p'$};
    \filldraw[black] (-1,-0.15) circle (0.5pt) node[anchor=east]{$T$};
    \filldraw[black] (1,-0.15) circle (0.5pt) node[anchor=west]{$T$};
  \end{tikzpicture}
  \caption{A cross-section of the region from which $K_1$ is excluded, which is rotationally symmetric about the line $qp'$.  Since $\sqrt{2}<3/2$, the disks overlap even when $p=q$.}
  \label{fig:balls}
\end{figure}
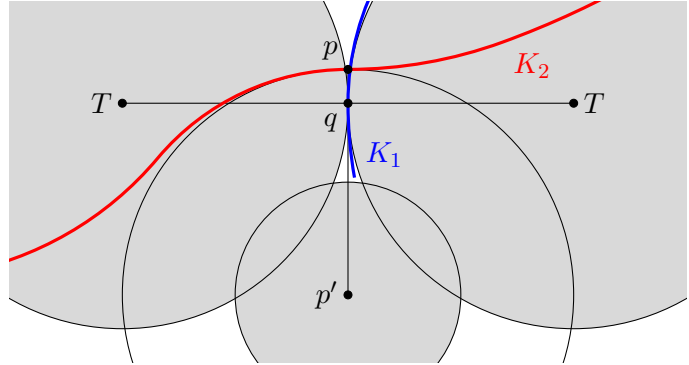

We first show that $\pi$ is an immersion.  Suppose not, so there is a point $q \in K_1$ such that $\pi'(q)=0$.  Let $p=\pi(q)$, let $T$ be the circle of radius $\epsi/2$ around $q$ in the normal plane to $K_1$ at $q$, and let $\mathbf v$ be a tangent vector to $K_1$ at $q$ which makes an obtuse or right angle with $pq$.  Finally, let $p'$ the point at distance $\epsi/2$ from $p$ along the ray from $q$ in the direction of $\mathbf v$.  By Lemma \ref{lem:connected}, $K_1$ cannot enter the $\epsi/2$-neighborhood of $T$.  Likewise, $K_2$ cannot enter the $\epsi/2$-neighborhood of $p'$, and therefore $K_1$ cannot enter the $\epsi/4$-neighborhood of $p'$.  These two neighborhoods together cut out a bounded region which $K_1$ cannot exit after it passes through $q$ in the direction of $p'$ (see Figure \ref{fig:balls}).

Now we have shown that $\pi$ is an immersion.  If $\pi$ is not injective, this means that $K_1$ winds multiple times around $K_2$.  Therefore we can find a ball $B$ of radius $\epsi/4$ centered at some point $q \in K_2$ such that $K_1 \cap B$ has at least two connected components, contradicting Lemma \ref{lem:connected}.
\end{proof}

\begin{refthm}{\ref{thick}}
  Let $L$ be any tame link which not isotopic to the unlink. Let $nL$ be the link consisting of $n$ copies of $L$ contained in disjoint balls. Then for every $\epsi>0$ and every $\epsi$-thick embedding of $nL$, we have $n=\exp(O(\epsi^{-3}))$.
\end{refthm}

\begin{proof}
Fix an $\epsi$-diagonal embedding of $nL$ with $\epsi$-thick components.

Tile $I^3$ by cubes of sidelength $\epsi/10$. After a tiny perturbation of the tiles, we can assume that the 2-faces of the tiles intersect the embedding transversely. To each component of each copy of $L$ associate a sequence of 2-faces which it crosses. Since each component is an embedded circle with an embedded $\epsi$-thick tubular neighborhood, it meets each cube in one segment. So this sequence is well-defined up to a cyclic permutation. There are $O(\exp(\epsi^{-3}))$ choices for such a sequence. We will show that two copies of a link cannot have the same associated sequence, and the theorem will follow.

Suppose for contradiction that two copies of our link, $L_1$ and $L_2$, have the same associated sequence of two-faces.  Consider the 3-manifold with boundary $M$ obtained by blowing up $I^3$  at $L_1$, that is by replacing $L_1$ by its unit normal bundle. Without loss of generality, since $L$ is not isotopic to the unlink, we can assume that the first component of $L_1$ does not bound a disc in $I^3 - L_1$. Let $K_i$ denote the first component of $L_i$. Since $L_1$ and $L_2$ have the same associated sequence, $K_1$ lies in the $\epsi/4$-neighborhood of $K_2$ and so we can apply Lemma \ref{knot_isotopy} to them. In particular there is an embedded annulus $A \subset I^3$ with $\partial A = K_1 \cup K_2$, and this can be lifted to an annulus $\tilde A \subset M$ with $\partial \tilde A=\tilde K_1 \cup K_2$. Since we assumed that $L_1$ and $L_2$ are separated by a sphere, we can find a piecewise linear map $F: D^2 \to M^{\circ}$ so that $F(\partial D^2) = K_2$. If we apply Dehn's lemma \cite[Theorem 1]{Papakyriakopoulos} to the disc $A \cup F(D^2)$, we obtain an embedded disc in $M$ with boundary $\tilde K_1$ whose interior lies in the interior of $M$. This means that $K_1$ bounds an embedded disc in $I^3 - L_1$, which is a contradiction. 
\end{proof}

\section{Background on Milnor invariants}

Milnor introduced his eponymous link invariants in his PhD thesis, which was published as \cite{Milnor,Milnor2}.  They can be thought of as higher-order generalizations of the linking number between two strands.  There are several ways to make sense of this.

One definition of the linking number is as follows: the linking number between curves $\gamma_1$ and $\gamma_2$ in $I^3$ is the class of $\gamma_1$ in the abelianization of $\pi_1(I^3 - \gamma_2)$.  Similarly, the Milnor invariants of an $m$-component link introduced in \cite{Milnor} measure the class of a loop $\gamma_1$ in a lower central series quotient of $\pi_1(I^3 - \gamma_2 - \cdots - \gamma_m)$.

The definition of the linking number between $\gamma_1$ and $\gamma_2$ can be rephrased using the presentation of the longitude of $\gamma_1$ in terms of the meridians of the two curves.  While the advantage of this more complicated formulation is not immediately clear, it generalizes to a larger class of link invariants introduced in \cite{Milnor2}.

Finally, one can understand the linking number as the coefficient of the cup product of cohomology classes in the complement Alexander dual to $\gamma_1$ and $\gamma_2$.  Similarly, Milnor invariants can be reformulated in terms of higher cohomology operations, an idea first suggested by Stallings \cite{Stall} and later developed by Turaev \cite{Turaev} and Porter \cite{Porter}.

We now give the precise definitions and the technical lemmas needed to make them work.

\subsection{General Milnor invariants}
Given a group $G$ define the lower central series by letting $\Gamma_1 G=G$, $\Gamma_2 G = [G, G]$, and for $q > 2$, letting $\Gamma_q G = [G, \Gamma_{q-1} G]$. Denote $Q_qG = G/\Gamma_q G$.  
In \cite{Milnor2}, Milnor extended a result of Chen and proved,

\begin{thm}[{\cite[Theorem 2]{Milnor2}}]
  If $L$ and $L'$ are two isotopic links in $I^3$, then for all $q \ge 1$,
  \[Q_q\pi_1(I^3 - L) \cong Q_q\pi_1(I^3 - L').\]
\end{thm}

The nontriviality of this stems from the fact that \emph{isotopy} here is not the ambient isotopy usually used as an equivalence relation on knots and links, but a homotopy through topological embeddings.  In the course of an isotopy in this sense, a knotted portion of a component may be unknotted by pulling it taut, changing the topology of the link complement.

Then Milnor gave an explicit presentation for this isotopy invariant of a link. Let $F_m$ be a free group on the variables $\{y_1, y_2 \ldots y_m\}$.  

\begin{thm}[{\cite[Theorem 4]{Milnor}}] \label{MilnorThm4}
  Suppose $L$ is an $m$-component link in $I^3$. For all $q \ge 1$, there is an isomorphism
  \[Q_q\pi_1(I^3 - L) \cong \frac{F_m}{\langle  \Gamma_q F_m, [y_i, w_i]:  1 \le i \le m \rangle},\]
  where the meridian of the $i^{th}$ loop gets mapped to $y_i$, and $w_i$ are words in $\{y_1, \ldots, y_m\}$ representing the conjugacy class of the $i^{th}$ longitude of $L$.
\end{thm}
Here the \emph{meridian} of a link component is a small loop that links with it, and the \emph{longitude} is a parallel loop whose linking number with the component is zero.

A Milnor invariant of order $q$ can be thought of as a measurement of how much
$Q_q \pi_1(I^3 - L)$ deviates from $Q_q F_m$. Let $\mathbb{Z}\{x_1, x_2, \ldots, x_m\}$ be the ring of formal power series over the integers generated by $m$ non-commuting variables. There is an injective homomorphism called the Magnus expansion,
$$M: F_m \to \mathbb{Z}\{x_1, x_2, \ldots, x_m\}$$
given by $M(y_i) = 1+x_i$ and $M(y_i^{-1}) = (1-x_i+x_i^2-x_i^3 + \cdots)$. For a word $w \in F_m$ and a multi-index $I = (i_1, i_2, \ldots, i_q)$ let $\mu_w(I)$ denote
coefficient of $x_{i_1}x_{i_2}\cdots x_{i_q}$ in the power series $M(w)$. In \cite{Magnus}, Magnus showed the following lemma.

\begin{lem}[{\cite{Magnus}}] Let $w$ be a word in $F_m$. Then
 $w \in \Gamma_{q+1}F_m$ if and only if $\mu_w(I) = 0$ for all multi-indices $I$ of size at most $q$.
\end{lem}

For a multi-index $I = (i_1, i_2, \ldots, i_q)$, Milnor defines $\mu(I) = \mu_{w_{i_q}}(i_1, i_2, \ldots, i_{q-1})$. Then he lets $\Delta(I)$ be the greatest common divisor of the set $\{\mu(I')\}_{I'}$ where $I'$ ranges over any multi-index obtained by removing one or more entries from $I$ and cyclically permuting the rest of the entries. The Milnor invariant $\overline \mu(I)$ is defined to be $\mu(I)$ modulo $\Delta(I)$.  We call $q$ the \emph{order} of the Milnor invariant; note that if all Milnor invariants of lower order are trivial, then $\overline\mu(I)$ lies in $\ZZ$.  Milnor showed that the $\overline \mu(I)$ are well-defined and, moreoever, are isotopy invariants.

\begin{thm}[{\cite[Theorem 5]{Milnor2}}] Suppose $L$ is an $m$-component link. Then $\overline \mu(I)$ is an isotopy invariant of $L$, for all multi-indices $I$ of size at most $m$. In particlar, if $L$ is isotopic to the trivial $m$-component link, then
$\overline \mu(I) = 0$ for all multi-indices $I$ of size at most $m$. 
\end{thm}
From the cited results, we can conclude the following:
\begin{thm} \label{MilnorSummary}
    Let $L$ be an $m$-component link, and suppose all Milnor invariants of $L$ of order $<q$ vanish.  Then $Q_q\pi_1(I^3 - L) \cong Q_qF_m$ via the map sending meridians to generators, and each $w_i \in \Gamma_{q-1}\pi_1(I^3 - L)$, where $w_i$ are elements corresponding in the conjugacy classes of the longitudes.
\end{thm}
\begin{proof}
    We see this by induction on $q$.  When $q=1$ there is nothing to check since all the groups are trivial.  Now suppose we have the result for $q$ and we know that $\overline \mu(I)=0$ for $I$ of length at most $q$.  Since we know that $Q_q\pi_1(I^3 - L) \cong Q_qF_m$, by Magnus' lemma it follows that each $w_i \in \Gamma_q\pi_1(I^3 - L)$.  Then for any $y_i$, $[y_i,w_i] \in \Gamma_{q+1}\pi_1(I^3 - L)$, and so by Theorem \ref{MilnorThm4}, $Q_{q+1}\pi_1(I^3 - L) \cong Q_{q+1}F_m$.  This gives the result for $q+1$.
\end{proof}

\subsection{Milnor groups and link homotopy} \label{S:link-htpy}
A \emph{link homotopy} is a deformation in which a strand of the link may go through itself but not through other strands.  More formally, a homotopy between two $n$-strand links $L_0$ and $L_1$ is a set of maps $\gamma^i_t:S^1 \times [0,1] \to I^3$, $i=1,\ldots,n$ such that $\{\gamma^i_0\}$ and $\{\gamma^i_1\}$ are embeddings with image $L_0$ and $L_1$, respectively, and for every $i \neq j$ and every $t \in [0,1]$, $\gamma^i_t$ and $\gamma^j_t$ have disjoint images.  Milnor studied invariants of link homotopy in \cite{Milnor}; see \cite{Kru} for another exposition of the main results.

We say that a link is \emph{homotopically trivial} if it is link homotopic to the unlink, and \emph{almost homotopically trivial} if all of its proper sublinks are link homotopic to the unlink.

Milnor showed that an invariant of link homotopy is the \emph{link group} or \emph{Milnor group}.  Let $L \subset I^3$ be an $n$-component link, and denote the meridian of the $i$th component by $m_i$.  Then its Milnor group $G(L)$ is
\[\pi_1(I^3 - L)/\langle\langle [m_i,gm_ig^{-1}] : i=1,\ldots,n,\; g \in \pi_1(I^3 - L) \rangle\rangle.\]
In particular, the Milnor group of the $n$-component unlink is the \emph{free Milnor group}, which we follow Freedman in denoting $FM_n$.

The group $FM_n$ is nilpotent of step $n$ and torsion-free; it can be thought of as a lattice in a nilpotent Lie group whose corresponding Lie algebra is spanned by iterated brackets of distinct generators, e.g.\ a basis for the Lie algebra of $FM_3$ is given by
\[X,\; Y,\; Z,\; [X,Y],\; [Y, Z],\; [Z, X],\; [X, [Y, Z]],\; [Z, [X, Y]],\]
where $X$, $Y$, and $Z$ are vectors in the direction of the generators of the group.  To demonstrate this identification, Milnor showed that $FM_n$ admits an injective Magnus expansion homomorphism
\[M:FM_n \hookrightarrow \mathcal R(x_1,\ldots,x_n)^\times\]
sending generators $m_i \mapsto 1+x_i$.  Here $\mathcal R(x_1,\ldots,x_n)^\times$ is the group of units of the ring $\mathcal R(x_1,\ldots,x_n)$ of polynomials with integer coefficients in non-commuting variables $x_1,\ldots,x_n$ modulo the ideal generated by monomials in which any variable occurs more than once.
\begin{thm}[{Milnor \cite[Theorem 8]{Milnor}}]
    An $n$-component link is homotopically trivial if and only if its Milnor group is $FM_n$.
\end{thm}
In particular, if an $n$-component $L=L_1 \cup \cdots \cup L_n$ is almost homotopically trivial, then the Milnor group $G(L')$ of $L'=L_2 \cup \cdots \cup L_n$ is $FM_{n-1}$.  Consider the element $\ell=[L_1] \in G(L')$; this is well-defined up to conjugation.  If $\ell=1$, then $L$ is the unlink.

Conversely, if $\ell$ has a nontrivial projection onto the quotient group $G_i \cong FM_{n-2}$ induced by forgetting about some strand $L_i$, $i \neq 1$, then $L - L_i$ is a homotopically nontrivial proper sublink of $L$.  Thus if $L$ is almost homotopically trivial, then $\ell$ must be in the kernel of all of these projections $p_i:FM_{n-1} \to FM_{n-2}$.  The intersection of these kernels is central in $G(L')$.  This can be most easily seen by using the Magnus expansion.  Applying the Magnus homomorphism to the domain and codomain of the map
\[FM_{n-1} \xrightarrow{\prod p_i} \prod_{i=1}^n FM_{n-2},\]
we get a restriction of the homomorphism
\[\mathcal R(x_1,\ldots,x_n) \to \prod_{i=1}^n \mathcal R(x_1,\ldots,\widehat{x_i},\ldots,x_n),\]
where the $i$th coordinate map sends the variable $x_i \mapsto 0$.  The kernel of this homomorphism consists exactly of polynomials of the form $1+a$, where $a$ is homogeneous of degree $n$.  Such polynomials are central in $\mathcal R(x_1,\ldots,x_n)^\times$, and so their preimages are central in $FM_{n-1}$.

Thus if $L$ is almost homotopically trivial, but homotopically nontrivial, then $\ell$ must be a nontrivial element of the center $Z(G(L'))$.  (A slightly more involved argument would show that this is the same as $\Gamma_{n-1}G(L')$.)  This gives us a complete, free abelian group--valued homotopy invariant of almost homotopically trivial links.  We will use this in the next section to bound the number of copies of such a link that may be $\epsi$-diagonally embedded.

\section{Sharp upper bound for homotopically nontrivial links}
In \cite[Theorem 6]{Free}, Freedman gives a non-optimal upper bound for $\epsi$-diagonally packing homotopically nontrivial links.  In this section, we improve this by proving that the lower bound from Theorem \ref{lower-bound} is sharp for these links.

\begin{refthm}{\ref{sharp-upper-bound}}
  Suppose $L$ is a homotopically nontrivial link with $m$ components. Let $nL$ be the link consisting of $n$ copies of $L$ contained in disjoint balls. Then for every $\epsi>0$ and every diagonal $\epsi$-embedding of $nL$, $n=\exp(O(\epsi^{-3}))$.
\end{refthm}

%New Version using referee's comments
\begin{proof}
Suppose first that $L$ has a proper sublink $L'$ which is homotopically nontrivial.  Then the desired result follows from the same result for $nL'$.  Therefore we may assume that $L$ is almost homotopically trivial.

We begin as before: tile $I^3$ by cubes of sidelength $\approx \epsi/2$, slightly perturbed if needed so that the tiling is transverse to $L$.  Since the cubes have diameter $<\epsi$, every tile has nontrivial intersection with at most one component of every copy of $L$.  We will show that for each collection of tiles $U$, there are $O(\epsi^{-3(m-1)})$ copies of $L$ so that the first component of each copy is contained in $U$ and the other components in $I^3 - U$. The theorem will follow since there are $\exp(O(\epsi^{-3}))$ ways to choose such a $U$.

Now consider any fixed $U$ which is a union of some of our tiles. Denote the first component of the $i^{th}$ copy of $L$ by $L_i^1$ and denote the union of all the other components of the $i^{th}$ of $L$ by $L_i'$. Let $I \subset \{1,2,\ldots,n\}$ be all the indices such that $L^1_i$ is contained in $U$. So for each $i \in I$ we have $L'_i \subset I^3-U$, since each cube intersects at most one component from each copy of $L$. We will bound $|I|$ by showing that the elements $[L^1_i] \in \pi_1(U)$ induce a large rank subgroup in a certain nilpotent group and bounding the rank of this group.

Let $FM_{m-1}$ be the free Milnor group on $(m-1)$ variables, and consider the homomorphisms
\[\rho_i: \pi_1(I^3 - L_i') \to FM_{m-1}\]
sending meridians to generators, which are well-defined and surjective by Theorem \ref{MilnorSummary}.  Observe that since the copies of $L$ lie in disjoint balls, for any $k \in I$ and $k \neq i$, $[L^1_k]$ is 0 in $\pi_1(I^3 - L'_i)$, and so $\rho_i([L^1_k])=0$. Next, define $G$ to be the subgroup of $\pi_1(U)$ generated by the images of all the $L_i^1$ for $i \in I$.  The inclusion $U \subset I^3 - L_i'$ induces a map $\nu_i:\pi_1(U) \to \pi_1(I^3 - L_i')$. By the previous observation the image of $G$ under the composition $\rho_i \circ \nu_i(G)$ is a cyclic group in $FM_{m-1}$ generated by $\rho_i[L^1_i]$. Since $FM_{m-1}$ is torsion-free, this group must be $\ZZ$. We can put together all these homomorphisms to get a map,
\[F: \pi_1(U) \to \bigoplus_{i \in I} FM_{m-1}\]
where the $i^{th}$ coordinate is defined by $\rho_i \circ \nu_i$. Note that $F(G) \cong \ZZ^{|I|}$, since $F(L^1_i)$ goes to an element with a nonzero $i^{th}$ coordinate and all other coordinates $0$. To bound $|I|$ in terms of $\epsi$, notice that $\pi_1(U)$ is composed of $O(\epsi^{-3})$ simply connected tiles and so $\pi_1(U)$ is generated by $O(\epsi^{-3})$ elements. Let $N = F(\pi_1(U))$, which is an $(m-1)$-step nilpotent group also generated by $O(\epsi^{-3})$ elements. For $q \ge 1$, $\Gamma_q N$ is generated by $O(\epsi^{-3q})$ elements (the exact number was computed by Witt \cite[p.~201]{Witt}).

Let $G_q=F(G) \cap \Gamma_q N$.  As $\Gamma_q FM_{m-1} / \Gamma_{q+1} FM_{m-1}$ is free abelian, so is $G_q/G_{q+1}$, so there is a splitting
\[G_q \cong G_q/G_{q+1} \oplus G_{q+1}\]
and by induction $F(G) \cong \bigoplus_{q=1}^{m-1} G_q/G_{q+1}$.  On the other hand,
\[\rank (G_q/G_{q+1}) \leq \rank (\Gamma_q N/\Gamma_{q+1} N) = O(\epsi^{-3q}).\]
Therefore
\[|I|=\rank F(G)=\sum_{q=1}^{m-1} \rank (G_q/G_{q+1})=O(\epsi^{-3(m-1)}). \qedhere\]
% Let $I_q \subset I$ be the collection of those indices $i$ so that $F(L^1_i)$ lands in $\Gamma_q N$ but not in $\Gamma_{q+1} N$, and let $G_q \le G$ be generated by $\{L_i^1\}_{i \in I_q}$. As $\Gamma_q FM_{m-1} / \Gamma_{q+1} FM_{m-1}$ is torsion-free, we have an injective homomorphism
% \[F(G_q) \to \Gamma_q N \to \Gamma_q N / \Gamma_{q+1} N.\]
% As $\Gamma_q N / \Gamma_{q+1} N$ is abelian we have
% \[|I_q| = \rank(F(G_q)) \le \rank(\Gamma_q N / \Gamma_{q+1} N) \le \rank(\Gamma_q N) = O(\epsi^{-3q}).\]
% And since $N$ is $(m-1)$-step nilpotent we conclude $|I| = \sum_{q=1}^{m-1} |I_q| = O(\epsi^{-3(m-1)})$.
\end{proof}

\section{Weak upper bound for a larger class of links}
\begin{refthm}{\ref{weak-upper-bound}}
  Suppose that $L$ is an $m$-component link whose Milnor invariants do not all vanish.  Let $q$ be the order of the first nontrivial Milnor invariant of $L$.  Then for every diagonal $\epsi$-embedding of $nL$, $n=\exp(O(\epsi^{-6(q-1)}))$.
\end{refthm}
For example, this gives a bound $n=\exp(O(\epsi^{-18}))$ when $L$ is the Whitehead link, which has a nontrivial Milnor invariant of type $(1,1,2,2)$.
\begin{proof}
Suppose we have a link $L$ with a first nontrivial Milnor invariant of order $q$, perhaps with repeating terms.  Without loss of generality, we can assume that this invariant depends nontrivially on the first component of the link.  Suppose we have a diagonal $\epsi$-embedding of $nL$.  As before, fix a tiling of $I^3$ by cubes of side length $\approx \epsi/2$, perhaps slightly perturbed so that they are transverse to the embedded link.  Let $L_i^j$ be the $j$th component of the $i$th copy of the link $L$.

Let $U \subset [0,1]^3$ be the union of some cells of our tiling, and let $I \subseteq \{1,\ldots,n\}$ be the set of indices for which the union of cells intersecting $L_i^1$ is $U$.  It follows that for $i \in I$, $L_i^j \cap U=\emptyset$ for $j \neq 1$.  As in the proof of \cite[Theorem 6]{Free}, we will bound $|I|$ by a pigeonhole argument: we will define a finite invariant associated to a loop in $U$ such that for every possible value of this invariant, there is at most one $i \in I$ such that the invariant takes this value on $L_i^1$.  Since there are $\exp(O(\epsi^{-3}))$ possible values of $U$, this will imply a bound on $n$.

We will assume in the rest of the proof that $I^3 - U^\circ$ is connected.  If it's not, then we can fix this by drilling $O(\epsi^{-3})$ thin tubes through $U$ that miss all the curves $L_i^1$.  This operation preserves all the necessary properties of $U$: its fundamental group as well as that of its complement is generated by $O(\epsi^{-3})$ generators, its interior contains each $L_i^1$, and its complement contains each $L_i^j$ for $j \neq 1$.

Write $L_i'=L_i^2 \cup \cdots \cup L_i^m$.  We know that for each $i$, the $L_i$ are isotopic, and the isomorphism induced by isotopy preserves the class in $Q_q\pi_1(I^3 - L_i)$ of the longitude $\ell_i$ of $L_i^1$, which we call $[\ell_i]$.  (This class is only well-defined up to conjugation, but we can pick a representative.)  Moreover, this class is nontrivial since $L_i$ has a nontrivial $q$th order Milnor invariant.  Since nilpotent groups are residually finite \cite{Hall}, we can choose an integer $p$ so that $[\ell_i]$ remains nontrivial in the group
\[Q_q^p\pi_1(I^3 - L_i)=Q_q\pi_1(I^3 - L_i)/\langle x^p : x \in Q_q\pi_1(I^3 - L_i)\rangle.\]

Conversely, for $i \neq j$, since the link $L_i^1 \cup L_j'$ is split, we have that
\[\pi_1(I^3 - L_i^1 - L_j') \cong \pi_1(I^3 - L_i^1) * \pi_1(I^3 - L_j'),\]
with the class of the longitude of $L_i^1$ contained in the left factor.  Therefore this class is trivial in $Q_q\pi_1(I^3 - L_i^1 - L_j') \cong \ZZ * Q_q\pi_1(I^3-L_j')$.

To summarize, let $\mathcal C$ be the set of embedded loops $C \subset U$ such that all Milnor invariants of order $<q$ vanish for $C \cup L_i'$ for all $i \in I$.  By Theorem \ref{MilnorSummary}, for every $i \in I$, there is an isomorphism
\[\alpha_i^C:Q_q^p\pi_1(I^3 - C - L_i') \xrightarrow{\cong} Q_q^pF_m,\]
which is well-defined up to conjugation, sending the meridians of the components to generators.  Pushing the longitude of $C$ forward along $\alpha_i^C$ for each $i$ gives a map
\[\alpha:\mathcal C \to \bigoplus_{i \in I} Q_q^pF_m,\]
which is injective on $\{L_i^1: i \in I\}$ by the discussion above.

Now suppose we find an invariant of a curve $C \in \mathcal C$, taking values in a finite set which depends only on $U$, which determines $\alpha(C)$.  Then $|I|$ is bounded by the number of possible values taken by this invariant.
\begin{lem}
  Such an invariant is given by: an element of the finite group $Q_q^p F_{r+1}$, a homomorphism $Q_q^p F_s \to Q_q^p F_{r+1}$, and a subgroup of $Q_q^pF_{r+1}$, where $r$ and $s$ are the numbers of generators of $H_1(U)$ and $H_1(\partial U)$, respectively.
\end{lem}
\begin{proof}
  By the van Kampen theorem, the homomorphisms
  \[\xymatrix{
    \pi_1(\partial U) \ar[r]^-{\beta_i} \ar[d]^{\ph_C} & \pi_1(I^3 - U^\circ - L_i') \ar[d] \\
    \pi_1(U - C) \ar[r] & \pi_1(I^3 - C - L_i')
  }\]
  induced by the inclusion maps form a pushout diagram, that is,
  \[\pi_1(I^3 - C - L_i') \cong \pi_1(U - C) *_{\pi_1(\partial U)} \pi_1(I^3 - U^\circ - L_i').\]
  After applying the functor $Q_q^p$ to each of the groups, this is no longer an isomorphism, but we get a map
  \[Q_q^p\pi_1(U - C) *_{Q_q^p\pi_1(\partial U)} Q_q^p\pi_1(I^3 - U^\circ - L_i') \to Q_q^p\pi_1(I^3 - C - L_i') \cong Q_q^pF_m\]
  which is clearly surjective.  To determine the class of the longitude of $C$ for every $i$, it suffices to describe its class in $Q_q^p\pi_1(U - C)$ and give enough information to specify the homomorphism $Q_q^p\pi_1(\partial U) \to Q_q^p\pi_1(U - C)$.  The main issue is that $\pi_1(U - C)$ depends on $C$.

  Consider the homomorphism
  \[\psi_C:F_{r+1} \to \pi_1(U - C)\]
  induced by sending the first $r$ generators $y_1,\ldots,y_r$ to curves $\gamma_1,\ldots,\gamma_r$ generating $H_1(U)$ and the last generator $y_0$ to a meridian of $C$.  While this homomorphism may not be surjective, it becomes surjective when we apply the functor $Q_q$.  This is because the abelianization $H_1(U - C)$ is spanned by $\psi_C(y_0),\ldots,\psi_C(y_r)$, and therefore any nilpotent quotient group of $\pi_1(U - C)$ is generated by these elements \cite[Lemma 5.9]{MKS}.  Now we can lift the homomorphism
  \[Q_q^p\ph_C:Q_q^p\pi_1(\partial U) \to Q_q^p\pi_1(U - C)\]
  to a homomorphism $\widetilde{\ph_C}:Q_q^pF_s \to Q_q^pF_{r+1}$ by fixing a set of generators for $H_1(\partial U)$.  We can also lift the class of the longitude of $C$ in $Q_q^p\pi_1(U - C)$ to an element $\widetilde w \in Q_q^pF_{r+1}$.

  Since $Q_q^p\psi_C$ is surjective, we can specify it (and thereby the group $Q_q^p\pi_1(U - C)$) by its kernel, a subgroup of $Q_q^pF_{r+1}$.

  We claim that $Q_q^p\psi_C \circ \widetilde{\ph_C}$ and $\widetilde w$ uniquely (up to conjugation) determine the class of the longitude of $C$ in $\alpha_i^CQ_q^p\pi_1(I^3 - C - L_i')$ for every $i$.  Suppose that we have two curves $C$ and $C'$ in $U$ equipped with homomorphisms $\ph_C:\pi_1(\partial U) \to \pi_1(U - C)$ and $\ph_{C'}:\pi_1(\partial U) \to \pi_1(U - C')$, respectively, and suppose that there is an isomorphism $\iota:Q_q^p\pi_1(U - C) \to Q_q^p\pi_1(U - C')$ such that $\iota \circ Q_q^p\ph_C=Q_q^p\ph_{C'}$.  Then by the functoriality of the pushout, $\iota$ induces an isomorphism
  \[\iota * \id:Q_q^p\pi_1(I^3 - C - L_i') \xrightarrow{\cong} Q_q^p\pi_1(I^3 - C' - L_i'),\]
  for any $i \in I$.  This isomorphism preserves the image of the meridians of components of $L_i'$ and sends the image of the meridian of $C$ to that of $C'$, therefore $\alpha_i^{C'} \circ (\iota * \id)=\alpha_i^C$.  Thus, for any $w \in Q_q^p\pi_1(U - C)$, $w$ and $\iota(w)$ induce the same element in $Q_q^p F_m$.
\end{proof}
It remains to bound the number of possible values of this invariant.  We have $r,s=O(\epsi^{-3})$.  It follows that the size of the group $Q_q^p F_{r+1}$, a $(q-1)$-step nilpotent group with $O(\epsi^{-3})$ generators and exponent $p$, is bounded by $p^{O(\epsi^{-3(q-1)})}$, and any subgroup is generated by $O(\epsi^{-3(q-1)})$ elements.  We can specify a homomorphism $Q_q^pF_s \to Q_q^p F_{r+1}$ by specifying the image of each generator, and we can specify a subgroup of $Q_q^p F_{r+1}$ by specifying a generating set.  Therefore we get
\[\lvert I \rvert \leq \bigl(p^{O(\epsi^{-3(q-1)})}\bigr)^{1+O(\epsi^{-3})+O(\epsi^{-3(q-1)})}=\exp(O(\epsi^{-6(q-1)})).\]
This implies a similar bound for $n$.
\end{proof}

\section{Dependence of constants on the link type}

In this section we explore a bit how the constants in our results depend on the complexity of the link $L$.  Namely, we compute some bounds in the case of a specific sequence of links: let $L(n)$ be the link consisting of $n$ fibers of the Hopf map $S^3 \to S^2$.  (In particular, every pair of strands in $L(n)$ forms a Hopf link.)

In the introduction of \cite{Free}, Freedman considered the following problem: given $\epsi>0$, which links can be embedded in $I^3$ (or $S^3$) so that every strand is $\epsi$-separated from every other?  In particular, Freedman conjectured that this can only be done for $L(n)$ when $n=O(\epsi^{-2})$.  We start by confirming this conjecture:
\begin{thm}
  Every embedding of $L(n)$ is at most $Cn^{-1/2}$-separated, for some constant $C>0$.
\end{thm}
This estimate is sharp: a $Cn^{-1/2}$-separated embedding can be obtained by taking the fibers under the Hopf map of a $Cn^{-1/2}$-separated set of points in $S^2$.
\begin{proof}
  Suppose we have an $\epsi$-separated embedding of $L(n)$ for some $\epsi>0$.  Since the $\epsi/2$-neighborhoods of each strand are disjoint, there is at least one strand, say $\gamma$, such that
  \[\vol(N_{\epsi/2}(\gamma)) \leq n^{-1}.\]
  Now we tile $I^3$ at scale $\epsi$ and let $N \subseteq N_{\epsi/2}(\gamma)$ be the set of tiles that intersect $\gamma$ nontrivially.  Then $N$ contains at most $C \epsi^{-3} n^{-1}$ tiles for a universal constant $C$, and therefore also at most $C\epsi^{-3} n^{-1}$ $1$-cells.  We can express the homology class of $\gamma$ in $H_1(N;\ZZ/2\ZZ)$ using a cellular cycle $z \in C_1(N;\ZZ/2\ZZ)$.  Since every edge has length at most $\epsi$, the length of $z$ is bounded by $C\epsi^{-2} n^{-1}$.
  
  Then in $I^3$, $z$ bounds a minimal chain $\Sigma$ of area less than $C\epsi^{-2}n^{-1}$.  After smoothing the corners of $z$ very slightly, $\Sigma$ can be assumed to be a $C^1$ surface by work of Allard \cite[5.3.21]{FedBk}.  Since $\gamma$ has linking number $1$ with every other strand of $L(n)$ and $z$ is homologous to $\gamma$ inside $N$, $\Sigma$ must intersect every strand $\gamma_i$ in a point $p_i$.  Since the strands are $\epsi$-separated, the points $p_i$ must be contained in disjoint $\epsi/2$-balls.  And since the surface $\Sigma$ is minimal and therefore nonpositively curved,
  \[\Area(B_{\epsi/2}(p_i) \cap \Sigma) \geq \pi(\epsi/2)^2.\]
  Therefore $\Area(\Sigma) \geq \pi n(\epsi/2)^2$.  It follows that
  \[\frac{\pi}{4} n\epsi^2 < C\epsi^{-2}n^{-1}\]
  and therefore
  \[\epsi < (4C/\pi)^{1/4}n^{-1/2}. \qedhere\]
\end{proof}
\begin{rmk}
  The above proof can also be recast in purely combinatorial terms, without using minimal surface theory or curvature.
\end{rmk}
We can now estimate how the number of copies of $L(n)$ that can be $\epsi$-diagonally packed depends on the values of $\epsi$ and $n$.
\begin{thm}
  Whenever $n\epsi<<1$, the number $k$ of copies of $L(n)$ which can be $\epsi$-diagonally packed in $I^3$ satisfies
  \[C_1\exp(C_2\epsi^{-3}n^{-2}) \leq k \leq C_3\exp\bigl(C_4\epsi^{-3}n^{-1}\log n\bigr)\]
  for constants $C_1,C_2,C_3,C_4$.
\end{thm}
The large difference between the two bounds comes about because the lower bound is governed by the number of crossings in $L(n)$, which is $2n^2$, while the upper bound is governed by the number of strands $n$.  It is not clear which of these is closer to the true answer.  Neither argument fully generalizes to links beyond $L(n)$, but the lower bound argument works for any link with an ``efficient'' diagram.
\begin{proof}
  Throughout the proof, we denote unspecified constants by $c_i$, $i=1,2,\ldots$.

  The lower bound comes from the construction in the proof of Theorem \ref{lower-bound} with slight modifications.  We start with a link diagram for $L(n)$ in $I^2$ in which crossings are $c_1n^{-1}$-separated, such as the one sketched in Figure \ref{fig:L_n}.
  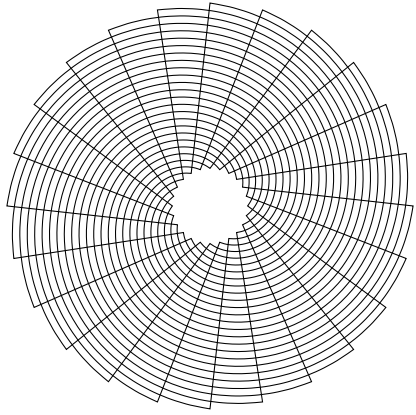
\begin{figure}
      \centering
      \begin{tikzpicture}[scale=0.5]
        \foreach \r in {0, 15,..., 345} {
          \draw [domain=0:330,variable=\t,smooth,samples=60]
            plot ({\t+\r}: {1+\t/75}) -- cycle;
        }
      \end{tikzpicture}
      \caption{An efficient diagram for $L(24)$.  The radial sections always cross above the spiral sections.}
      \label{fig:L_n}
  \end{figure}
  We realize this diagram in $I^3$ so that all crossings are realized by a strand near the top and a strand near the bottom.  Then at each crossing we fix a box around the lower crossing point whose dimensions are $c_2n^{-1} \times c_2n^{-1} \times 1/2$, so that it fits $c_3\epsi^{-3}n^{-2}$ cubes of side length $7\epsi$.  (Here we are using the assumption that $n\epsi<<1$.)  We then continue with the proof of Theorem \ref{lower-bound}, sending each strand parallel to the undercrossing through these $c_3\epsi^{-3}n^{-2}$ boxes and threading the strands parallel to the overcrossing around them.  This shows that
  \[k \geq 2^{c_3\epsi^{-3}n^{-2}}.\]

  To get the upper bound, we note that the strands of $L(n)$ all have identical topological roles, and therefore, for each $i$, we can assume that the strands $L^1_i,\ldots,L^n_i$ of the $i$th copy of $L(n)$ in $kL(n)$ are ordered so that
  \[\vol(N_\epsi(L^1_i)) \leq \vol(N_\epsi(L_2^i)) \leq \cdots \leq \vol(N_\epsi(L^n_i)).\]
  In particular, the $\epsi$-neighborhood of $L^1_i$ has volume at most $n^{-1}$.  Therefore the union $U_i$ of tiles which intersect nontrivially with $L_i^1$ consists of at most $c_6\epsi^{-3}/n$ tiles.  Using the easy bound ${a \choose b}<\left(\frac{ae}{b}\right)^b$, the number of distinct $U_i$ is at most
  \[N_U={c_6\epsi^{-3} \choose c_6\epsi^{-3}/n} < (ne)^{c_6\epsi^{-3}/n}=\exp\bigl(c_6\epsi^{-3}{\textstyle \frac{\log n}{n}}\bigr).\]
  The proof of Theorem \ref{sharp-upper-bound} implies that
  \[k=O(\epsi^{-3})N_U,\]
  and the $\epsi^{-3}$ factor can be absorbed into the constant $C_4$ in the exponent, proving the upper bound.
\end{proof}

\bibliographystyle{amsplain}
\bibliography{link-packing}
\end{document}